\documentclass[12pt]{amsart}

\usepackage[paperwidth=8.5in,paperheight=11in,inner=1in,outer=1in,top=1.2in,bottom=1.5in]{geometry}

\usepackage{amsmath}
\usepackage{amsthm}
\usepackage{amssymb}
\usepackage{tikz}
\usetikzlibrary{matrix}
\usepackage{pgfplots}
\usepackage{mathrsfs}
\usepackage[colorlinks=false]{hyperref}
\usepackage{paralist}

\theoremstyle{plain}
\newtheorem{Theorem}{Thm}[section]
\newtheorem{Thm}[Theorem]{Theorem}
\newtheorem{thm}[Theorem]{Theorem}

\newtheorem{lem}[Theorem]{Lemma}
\newtheorem{Cor}[Theorem]{Corollary}

\newtheorem{Prop}[Theorem]{Proposition}
\newtheorem{prop}[Theorem]{Proposition}

\newtheorem*{Question*}{Question}

\newtheorem*{Thm*}{Theorem}
\newtheorem*{thm*}{Theorem}
\newtheorem*{Cor*}{Corollary}
\theoremstyle{definition}
\newtheorem{Def}[Theorem]{Definition}
\newtheorem{Rem}[Theorem]{Remark}

\newtheorem{Exm}[Theorem]{Example}

\makeatletter
\renewenvironment{proof}[1][\proofname]{\par
\pushQED{\qed}%
\normalfont \topsep6\p@\@plus6\p@\relax
\trivlist
\item[\hskip\labelsep
\scshape
#1\@addpunct{.}]\ignorespaces
}{%
\popQED\endtrivlist\@endpefalse
}
\makeatother

\newcommand\mbb{\mathbb}

\newcommand\mcal{\mathcal}

\newcommand\A{\mbb{A}}
\newcommand\C{\mbb{C}}
\newcommand\N{\mbb{N}}
\renewcommand\P{\mbb{P}}

\newcommand\R{\mbb{R}}

\newcommand\V{\mcal{V}}

\newcommand\ul{\underline}

\DeclareMathOperator\lspan{span}
\DeclareMathOperator\tint{int}

\DeclareMathOperator\conv{conv}

\DeclareMathOperator\rk{rank}
\DeclareMathOperator\tr{tr}
\DeclareMathOperator\diag{diag}
\DeclareMathOperator\gr{Gr}

\newcommand\binomial{\binom}

\newcommand\ideal[1]{\langle #1 \rangle}
\newcommand\scp[1]{\langle #1 \rangle}

\newcommand\sym[1][n]{\mcal{S}^{#1}}
\newcommand\psd[1][n]{\mcal{S}^{#1}_{+}}
\newcommand\pd[1][n]{\mcal{S}^{#1}_{++}}

\newcommand\mlt{maximum likelihood threshold}
\newcommand\MLE{maximum likelihood estimator}
\newcommand\Grank{Generic completion rank}
\newcommand\grank{generic completion rank}
\DeclareMathOperator\ml{mlt}
\DeclareMathOperator\gcr{gcr}

\begin{document}
\title{Maximum Likelihood Threshold and Generic Completion Rank of Graphs}
\author{Grigoriy Blekherman}
\author{Rainer Sinn}

\begin{abstract}
The minimum number of observations such that the maximum likelihood estimator in a Gaussian graphical model exists with probability one is called the maximum likelihood threshold of the underlying graph $G$. The natural algebraic relaxation is the generic completion rank introduced by Uhler. We show that the maximum likelihood threshold and the generic completion rank behave in the same way under clique sums, which gives us large families of graphs on which these invariants coincide. On the other hand, we determine both invariants for complete bipartite graphs $K_{m,n}$ and show that for some choices of $m$ and $n$ the two parameters may be quite far apart. In particular, this gives the first examples of graphs on which the maximum likelihood threshold and the generic completion rank do not agree.
\end{abstract}

\maketitle

\section*{Introduction}
Gaussian graphical models are an important class of statistical models widely used in system biology and bioinformatics \cite{MR2064941, krumsiek2011gaussian, schafer2005empirical}. In such applications, one is often given data from few observations compared to the number of random variables. The maximum likelihood estimator is computed to describe the interaction between different variables, e.g. genes. It is well known that the {\MLE} exists with probability one if the number of observations is at least as large as the number of variables. However, with fewer observations the maximum likelihood estimator may fail to exists with high probability. Therefore, it is an important problem to find the minimal number of observations such that the sample covariance matrix has a maximum likelihood estimator with probability one in terms of the graph $G$ underlying the Gaussian graphical model. This number was studied in \cite{BuhlMR1241392}, \cite{UhlerMR3014306} and following \cite{GrossSullivant} we call it the \textit{maximum likelihood threshold} of $G$, denoted $\ml(G)$. See \cite{LaMR1419991} and \cite{WhMR1112133} for an introduction to Gaussian graphical models.

Gaussian graphical models were introduced by Dempster in \cite{Dempster}. In his seminal paper, he showed that existence of the {\MLE} is equivalent to a positive definite matrix completion problem. Using this reformulation, Buhl proved upper and lower bounds on the maximum likelihood threshold in  \cite{BuhlMR1241392}. Uhler in \cite{UhlerMR3014306} introduced an algebraic relaxation of the maximum likelihood threshold which we will call the \textit{generic completion rank} of $G$, which provides an upper bound on the \mlt. Gross and Sullivant made a connection between generic completion rank and combinatorial rigidity theory to prove upper bounds on the generic completion rank, and by extension on the maximum likelihood threshold of $G$. However, the work of Uhler and Gross-Sullivant left open the question of whether the maximal likelihood threshold and the generic completion rank are different for any graph $G$, as they agree in the small examples where both could be computed.

 In this paper, we show that the maximal likelihood threshold and the generic completion rank agree under \textit{clique sums} of graphs. Therefore, large classes of graphs on which these two parameters agree can be built using known cases of equality, see Corollary~\ref{cor:wheels}. However, we also compute both the maximal likelihood threshold and the generic completion rank of complete bipartite graphs $K_{m,n}$. It turns out that these quantities can be quite different in this case. This provides the first examples of graphs where the two parameters do not agree. This phenomenon is also interesting for practical purposes, as it shows that maximum likelihood estimator may exist with probability one with a much smaller number of observations than suggested by the generic completion rank.

We make several reinterpretations of the maximal likelihood threshold and the generic completion rank which we think are interesting in themselves, and we relate generic completion rank of a symmetric matrix, to the non-symmetric case. This shows that the generic completion rank is an important parameter in itself, in need of further study and understanding. We now review previous work and our results in detail.

\vspace{1em}
\noindent \textbf{Main Results.} We write $\sym$ for the vector space of real symmetric $n\times n$ matrices and $\psd$ for the cone of positive semidefinite symmetric matrices. Let $G=(V,E)$ be a simple graph on $n$ vertices labeled by the integers $[n] = \{1,2,\dots,n\}$. Given $G$, we have a projection $\pi_G$ from the real vector space $\sym$ to the real vector space $V = \R^n \oplus \R^E$, where the coordinates in the first summand are indexed by vertices of $G$ and in the second summand by edges of $G$. It is defined as
\[
\pi_G( a_{ij}) = (a_{ii} \colon i\in [n]) \oplus (a_{ij}\colon \{i,j\}\in E).
\]
The image of this projection is a partially specified symmetric matrix and we call it a \emph{$G$-partial matrix}. The existence of a unique maximum likelihood estimator of a $G$-partial matrix $\pi_G(A)$ in a Gaussian graphical model is equivalent to the existence of a positive definite matrix completion of $\pi_G(A)$ by a seminal result of Dempster~\cite{Dempster}.

Therefore the question of computing the maximal likelihood threshold is equivalent to the following matrix completion problem:

\begin{center} What is the minimal number $r$ such that for a generic positive semidefinite matrix $A$ of rank $r$, the projection $\pi_G(A)$ has a positive definite completion? \end{center}

Buhl \cite{BuhlMR1241392} proved bounds on $\ml(G)$ in terms of well-known graph parameters of $G$. Let $\omega(G)$ be the clique number of $G$ and $\tau(G)$ be the tree-width of $G$. Then we have the following:
\[
\omega(G) \leq \ml(G)\leq \tau(G)+1.
\]
However, the clique number and the tree-width of a graph may be far apart and these bounds are often far from sharp. In \cite{UhlerMR3014306}, Uhler introduced an algebraic relaxation of maximum likelihood threshold:

The \emph{\grank} of the graph $G$ is the smallest integer $r$ such that $\dim(\pi(V_r)) = \dim(V)$
where $V_r\subset \sym$ denotes the variety of symmetric matrices of rank at most $r$, see Definition~\ref{def:gcr}.
The {\grank} of $G$ is called the rank of $G$ in \cite{GrossSullivant}.

Uhler \cite{UhlerMR3014306} showed that 
\[
\ml(G) \leq \gcr(G),
\]
and, using this inequality, was able to compute the maximal likelihood threshold in some examples. 

Gross and Sullivant \cite{GrossSullivant} made a fruitful connection between generic completion rank of $G$ and combinatorial rigidity theory to prove some very interesting results. For instance, they show that the generic completion rank (and therefore the maximum likelihood threshold) of any planar graph is at most $4$. In all of the examples where both parameters could be computed, there is equality between the maximal likelihood threshold and the generic completion rank. For non-symmetric completion problems considered in Section \ref{sec:non-sym}, Bernstein recently gave a combinatorial classification of bipartite graphs whose {\grank} is two using tropical geometry tools \cite{Bernstein}.

We first show that maximal likelihood threshold and generic completion rank behave in the same way under clique sums. 
Recall that given two simple graphs $G_1 = (V_1,E_1)$ and $G_2=(V_2,E_2)$ such that the induced subgraphs on $V_1\cap V_2$ are cliques in $G_1$ and $G_2$, respectively, their \emph{clique sum} is the graph $(V_1\cup V_2,E_1\cup E_2)$. We denote it by $G_1\boxtimes G_2$. We show the following theorem:
\begin{thm*}[Theorem~\ref{thm:rankcliquesum} and Theorem~\ref{thm:mltcliquesum}]
For any simple graphs $G_1$ and $G_2$ $$\ml(G_1\boxtimes G_2)=\max\{\ml(G_1),\ml(G_2)\} \quad \text{and} \quad \gcr(G_1\boxtimes G_2)=\max\{\gcr(G_1),\gcr(G_2)\}.$$
\end{thm*}

Thus, we can build large families of graphs on which the two parameters coincide. On the other hand, we compute the generic completion rank and the maximum likelihood threshold of complete bipartite graphs:

\begin{Thm*}
Let $n\geq m\geq 2$ and set $M$ to be the smallest $k$ such that $\binomial{k+1}{2}\geq m+n$. 
\begin{compactenum}
\item The {\grank} of $K_{m,n}$ is $m+1$ if $n>\binom{m}{2}$. It is $m$ if $n\leq \binom{m}{2}$ (Theorem~\ref{thm:rankbipartite}).
\item The {\mlt} of the complete bipartite graph $K_{m,n}$ is the minimum between $M$ and $m+1$ (Theorem \ref{thm:bipartite}).
\end{compactenum}
\end{Thm*}

Note that the {\grank} of $K_{m,n}$ is $m$ or $m+1$, depending on the size of the larger part $n$ compared to $m$, and therefore, this is an upper bound on the {\mlt}. The above minimum is $M$ for $n\leq \binomial{m}{2}$ and $m+1$ for $n>\binomial{m}{2}$. Therefore we have the following Corollary:

\begin{Cor*}[Corollary~\ref{cor:gcrvsmlt}]
\begin{compactenum}[(a)]
\item The {\mlt} of $K_{5,5}$ is $4$, whereas the {\grank} is $5$.
\item Asymptotically, the {\mlt} of $K_{m,m}$ is of the order $\sqrt{m}$. 
\item The {\mlt} of $K_{m,n}$ is equal to the {\grank} of $K_{m,n}$ for $m=2,3,4$ and whenever $n > \binom{m}{2}$. 
\end{compactenum}
\end{Cor*}
The case $m=2$ in (c) was known before, see \cite[Proposition~4.2]{UhlerMR3014306}.

Another notion related to the {\grank}, called the Gaussian rank of a graph, was introduced by Ben-David in \cite{BenDavid}. The case of complete bipartite graphs shows also that the {\mlt} can be far from the Gaussian rank of a graph because the Gaussian rank of $K_{m,n}$ is its treewidth, i.e.~$m+1$, by \cite[Theorem~1.1]{BenDavid}.


In the end of this paper, we relate the {\grank} of bipartite graphs to the {\grank} for the general matrix completion problem: Given a partially specified $m\times n$ matrix, we form a bipartite graph on the row and column indices by adding the edge $\{i,j\}$ if and only if the $(i,j)$th entry of the partial matrix is given. We show that the {\grank} of the graph obtained this way is (up to an additive constant of at most $1$) equal to the {\grank} in the general case, see Proposition~\ref{prop:matrixcompletion}. We take this as an indication that the {\grank} of bipartite graphs is an important graph parameter.

We end this section with an open question of whether the maximum likelihood threshold can be used to provide an upper bound for the \grank.  As explained above for complete bipartite graphs $K_{m,m}$, the {\grank} is a quadratic function of \mlt. But it is not known whether there exist graphs with bounded \mlt and arbitrarily high \grank.

\begin{Question*}
Does there exist a function $f:\N \rightarrow \N$ such that for all $k \in \N$: $$\operatorname{mlt} (G) \leq k \,\,\,\, \text{implies} \,\,\,\, \operatorname{gcr} (G) \leq f(k).$$
\end{Question*}

\textit{Acknowledgements.} We would like to thank Seth Sullivant for helpful comments on an earlier version. The authors were partially supported by NSF grant DMS-1352073.

\section{{\Grank} and \mlt}
We write $\sym$ for the vector space of real symmetric $n\times n$ matrices and $\psd$ for the cone of positive semidefinite symmetric matrices. We denote the interior of $\psd$ by $\pd$; $\pd$ is the cone of positive definite matrices. Let $G=(V,E)$ be a simple graph on $n$ vertices labeled by the integers $[n] = \{1,2,\dots,n\}$. Given $G$, we have a projection $\pi_G$ from the real vector space $\sym$ of real symmetric $n\times n$ matrices to the real vector space $V = \R^n \oplus \R^E$, where the coordinates in the first summand are indexed by vertices of $G$ and in the second summand by edges of $G$. It is defined as
\[
\pi_G( a_{ij}) = (a_{ii} \colon i\in [n]) \oplus (a_{ij}\colon \{i,j\}\in E).
\]
We think of a vector in the image of this projection as a partially specified symmetric matrix and call it a \emph{$G$-partial matrix}. We are interested in the positive semidefinite matrix completion problem, i.e.~we would like to describe the image of the cone of positive semidefinite real symmetric matrices $\psd$ under the coordinate projection $\pi_G$, which is the convex cone of $G$-partial matrices with a positive semidefinite completion. We write $\Sigma_G$ for $\pi_G(\psd)$.

\begin{Def}\label{def:gcr}
The \emph{\grank} of the graph $G$ is the smallest integer $r$ such that the projection $\pi_G$ restricted to the variety of matrices of rank at most $r$ is dominant. In other words,
\[
\dim(\pi(V_r)) = \dim(V)
\]
where $V_r\subset \sym$ denotes the variety of symmetric matrices of rank at most $r$.
\end{Def}
The {\grank} of $G$ is called the rank of $G$ in \cite{GrossSullivant}.
By Chevalley's Theorem \cite[Exercise II.3.19]{HarMR0463157}, the set $\pi(V_r)$ is constructible in the Zariski topology. This implies that almost every $G$-partial matrix will have a completion with complex entries of rank equal to the {\grank} of $G$, which motivates the name.

We will give equivalent definitions from an algebro-geometric point of view. We fix the usual trace inner product on the real vector space of symmetric matrices, i.e.~$\scp{A,B} = \tr(AB)$ for $A,B\in \sym$. The orthogonal complement of the kernel of $\pi_G$ is the coordinate subspace of matrices whose $(i,j)$th entry is $0$ whenever $\{i,j\}$ is not an edge. We denote this linear space by $L_G$.

We interpret the kernel of the projection $\pi_G$ as a linear space of quadratic forms by the one-to-one correspondence $Q_M = (x_1,\dots,x_n) M (x_1,\dots,x_n)^t$; it is the span of the quadratic forms $x_ix_j$ for all non-edges $\{i,j\}$ of the graph. We write $I_G$ for the ideal generated by $\ker(\pi_G)$. This is a square-free monomial ideal and the Stanley-Reisner ideal of the clique complex of the graph. It defines an algebraic set $X_G=\V(I_G)\subset \P^{n-1}$, which is a subspace arrangement. In fact, $X_G$ is the union of all subspaces $\P(\lspan\{e_i\colon i\in K\})$, where $K\subset G$ is a clique and $e_i$ are the standard basis vectors in $\R^n$. We write $R$ for the homogeneous coordinate ring $\R[x_1,\dots,x_n]/I_G$ of $X_G$, which is graded by the total degree grading. We write $R_i$ for the homogeneous part of degree $i$. In this geometric setup, the convex cone $\pi_G(\psd)$ is a subset of $R_2$, namely the cone of sums of squares on $X_G$, hence the notation $\Sigma_G$, see \cite[Section~6.1]{BlekhermanSinnVelasco}.

In this paper, a \emph{linear series} on $X_G$ is a linear subspace of $R_1$. Using these notations, we give two equivalent definitions of the \grank.
We write $\scp{S}$ for the ideal generated by a subset $S\subset R$ and $\scp{S}_i$ for the homogeneous part of degree $i$ of a homogeneous ideal.
\begin{Prop}\label{prop:rankdefs}
The integers defined in the following statements are equal to the {\grank} of $G$.
\begin{compactenum}
 \item The smallest $k$ such that there exists a linear series $W\subset R_1$ of dimension $k$ on $X_G$ with the property that $\ideal{W}_2 = R_2$.
 \item The smallest $k$ such that there exists a linear series $W\subset R_1$ of dimension $k$ on $X_G$ which is not contained in the kernel of a matrix in the orthogonal complement of $\ker(\pi_G)$.
\end{compactenum}
\end{Prop}

\begin{proof}
The {\grank} is equal to the integer in (1) by Terracini's Lemma, see Flenner-O'Carroll-Vogel \cite[Proposition~4.3.2]{FlOCaVoMR1724388}. Indeed, the equality $\dim(\pi_G(V_r)) = \dim(V)$ holds if and only if the differential at a generic point on $V_r$ is surjective by Generic Smoothness \cite[Corollary III.10.7 and Proposition III.10.4]{HarMR0463157}. Terracini's Lemma says that the tangent space to $V_r$ at a generic point $M = \ell_1^2 + \dots + \ell_r^2$ is the degree $2$ part of the homogeneous ideal generated by $\ell_1,\dots,\ell_r$. So the differential is surjective if and only if $T_M V_r + \ker(\pi_G) = \sym$, which is in turn equivalent to $\ideal{\ell_1,\dots,\ell_r}_2 = R_2 \cong \sym/\ker(\pi_G)$.

We now prove that the two integers defined in (1) and (2) are equal. By duality in linear algebra, we have $\ideal{W}_2 = R_2$ for a linear series $W\subset R_1$ if and only if there is no linear functional $\ell\in R_2^\ast$ that vanishes on $\ideal{W}_2$. To a linear functional $\ell\in R_2^\ast$, we associate the quadratic form $Q_\ell\colon R_1\to \R$, $Q_\ell(f) = \ell(f^2)$; the representing matrix of $Q_\ell$ with respect to the monomial basis is the moment matrix or middle Catalecticant of $\ell$. Then $\ell$ vanishes on $\ideal{W}_2$ if and only if $W$ is in the kernel of the moment matrix of $\ell$. With the trace inner product, we identify $R_2^\ast$ with the orthogonal complement of $\ker(\pi_G)$, i.e.~$R_2^\ast \cong L_G$. So there exists a linear series $W\subset R_1$ of dimension $k$ such that $\ideal{W}_2 = R_2$ if and only if there is no matrix in the orthogonal complement of $\ker(\pi_G)$ that contains $W$ in its kernel. 
\end{proof}

Using these equivalent characterizations of the \grank, we can recover e.g.~\cite[Theorem~3.2]{GrossSullivant} by a simple dimension count.
\begin{Thm}\label{thm:dimcount}
Let $G$ be a simple graph on $n$ vertices and suppose that the {\grank} of $G$ is $r$. Then $G$ has at most $n(r-1) - \binom{r}{2}$ edges.
\end{Thm}

\begin{proof}
The number of edges of $G$ can be expressed in terms of the kernel of $\pi_G$, namely $n + \# E + \dim(\ker(\pi_G)) = \binom{n+1}{2}$. The fact that the {\grank} of $G$ is $r$ implies that $\dim(V_r) + \dim(\ker(\pi_G)) \geq \dim(\sym[n])$. The dimension of the variety of symmetric $n\times n$ matrices of rank at most $r$ is $r n - \binom{r}{2}$, so we get 
\[
\left[rn-\binom{r}{2}\right] + \left[\binom{n}{2} - \#E\right] \geq \binom{n+1}{2}.
\]
The claimed inequality follows by simplification of this one.
\end{proof}

We now turn to the interior of $\Sigma_G$, which is the set of all $G$-partial matrices that have a positive definite completion, see \cite[Proposition~5.5]{CLRMR1327293} or \cite[Lemma~1.5]{CPSV}.

The set of all symmetric matrices $A$ of rank $r$ such that $\pi_G(A)\in \tint(\Sigma_G)$ is a semialgebraic subset of the variety $V_r$ of matrices of rank at most $r$. In the following, we say that \emph{almost all} positive semidefinite matrices of rank $r$ map to $\tint(\Sigma_G)$ if the set of positive semidefinite matrices of rank $r$ mapping to the boundary of $\Sigma_G$ is contained in a proper algebraic subvariety of $V_r$. In other words, a positive semidefinite matrix of rank $r$ will map to the interior of $\Sigma_G$ with probability $1$.
\begin{Def}
The \emph{maximum likelihood threshold} of $G$ is the smallest integer $r$ such that for almost all positive semidefinite matrices $A$ of rank $r$ there exists a positive definite matrix $P$ with 
\[
\pi_G(A) = \pi_G(P).
\]
\end{Def}

As for the {\grank} in Proposition~\ref{prop:rankdefs}, we will also give equivalent definitions from an algebro-geometric point of view that look very similar.
\begin{Prop}\label{prop:mltdefs}
The integers defined in the following statements are equal to the {\mlt} of $G$.
\begin{compactenum}
 \item The smallest $k$ such that the vector space $\ideal{W}_2 + \ker(\pi_G) \subset \sym$ contains a positive definite matrix for a generic linear series $W\subset R_1$ of dimension $k$.
 \item The smallest $k$ such that a generic linear series $W\subset R_1$ of dimension $k$ is not in the kernel of a positive semidefinite matrix in the orthogonal complement of $\ker(\pi_G)$.
\end{compactenum}
\end{Prop}

\begin{proof}
As before, we first show that our definition of the {\mlt} is equivalent to (1). First, let $A = \ell_1^2 + \dots + \ell_r^2$ be a generic positive semidefinite matrix of rank $r$ such that $\pi_G(A) = \pi_G(P)$ for a positive definite matrix $P$. Set $W = \lspan\{\ell_1,\dots,\ell_r\}\subset R_1$. Then $P\in \ideal{W}_2 + \ker(\pi_G)$. On the other hand, let $W = \lspan\{\ell_1,\dots,\ell_r\}\subset R_1$ be a linear series such that $\ideal{W}_2 + \ker(\pi_G)$ contains a positive definite matrix. Then the $G$-partial matrix $\pi_G(A)$ for $A = \ell_1^2 + \dots + \ell_r^2$ has a positive definite completion. We need to show that this holds for generic matrices. The choice of basis $\{\ell_1,\dots,\ell_r\}$ of $W$ corresponds to the choice of a point $\ell_1^2 + \dots + \ell_r^2$ in the relative interior of a face of $\psd$, so that generic linear series $W\subset R_1$ of dimension $r$ correspond to generic positive semidefinite matrices of rank $r$, see \cite[Section~II.12]{BarMR1940576}.

We now show that the integers in (1) and (2) are equal. By the separation theorem in convex geometry and the fact that the cone of positive semidefinite matrices is self-dual with respect to the trace inner product \cite[Exercise~IV.5.3.2]{BarMR1940576}, we know that the vector space $\ideal{W}_2 + \ker(\pi_G)$ contains a positive definite matrix if and only if there is no positive semidefinite matrix in the orthogonal complement $L_G$ that contains $W$ in its kernel. This shows that the integers in (1) and (2) are the same.
\end{proof}

The above reformulations of the {\grank} and the {\mlt} of a graph make it clear that the {\grank} is the natural algebraic relaxation of the {\mlt}, a fact first observed by Uhler~\cite{UhlerMR3014306}.
\begin{Cor}
The {\mlt} of a graph is bounded above by its {\grank}.
\end{Cor}

\begin{proof}
This follows from the equivalent characterizations of the two invariants in Propositions~\ref{prop:rankdefs}(1) and \ref{prop:mltdefs}(1).
\end{proof}

We will study how {\grank} and {\mlt} behave under graph operations. It is straightforward to see that both invariants are non-increasing with respect to edge and vertex deletion. However, unlike a related parameter of Gram dimension introduced in \cite{LaVaMR3006041}, the properties $\operatorname{mlt} (G) \leq k$ and $\operatorname{gcr} (G) \leq k$ are not minor-closed as the following example shows.
\begin{Exm}\label{exm:minorclosed}
Start with the complete graph $K_4$ on $4$ vertices and divide every edge so that  we get the graph $G$ in Figure~\ref{fig:minorclosed}.

\begin{figure}[h]
\centering
\begin{tikzpicture}
  [place/.style={circle,draw,thick,inner sep = 2pt, minimum size = 10pt}]
  \node[fill=black] at (0,0) [place] (1) {};
  \node at (2,0) [place] (2) {};
  \node[fill=black] at (4,0) [place] (3) {};
  \node at (1,1) [place] (4) {};
  \node at (3,1) [place] (5) {};
  \node at (1,2) [place] (6) {};
  \node[fill=black] at (2,2) [place] (7) {};
  \node at (3,2) [place] (8) {};
  \node at (2,3) [place] (9) {};
  \node[fill=black] at (2,4) [place] (10) {};
\draw (1) -- (2) -- (3) -- (8) -- (10) -- (6) -- (1);
\draw (1) -- (4) -- (7) -- (9) -- (10);
\draw (3) -- (5) -- (7);
\end{tikzpicture}
\caption{The graph in Example~\ref{exm:minorclosed}.}
\label{fig:minorclosed}
\end{figure}
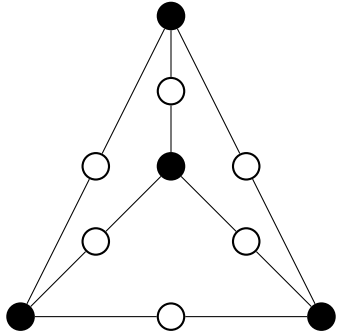

By construction, $K_4$ is a minor of $G$. The {\grank} and {\mlt} of $K_4$ are $4$. By inductively deleting vertices of degree $2$, we end up with an empty $3$-core, which shows that the {\grank} of $G$ is at most $3$ by \cite[Proposition~3.5]{GrossSullivant} (see also \cite[Theorem~3.7]{GrossSullivant}).
\end{Exm}

From the geometric point of view, we can understand the clique sum of graphs well.
\begin{Def}
 Let $G_1 = (V_1,E_1)$ and $G_2=(V_2,E_2)$ be two simple graphs such that the induced subgraphs on $V_1\cap V_2$ are cliques in $G_1$ and $G_2$, respectively. Their \emph{clique sum}, denoted by $G_1\boxtimes G_2$, is the graph $(V_1\cup V_2,E_1\cup E_2)$.
\end{Def}

The geometric operation corresponding to the clique sum of graphs is referred to as linear join in the algebraic geometry literature, see~\cite{EiGrHuPoMR2275024}.
\begin{Def}
Let $X_1$ and $X_2$ be two subspace arrangements in $\P^{n-1}$. We say that $X = X_1\cup X_2$ is \emph{linearly joined} if $X_1\cap X_2 = \langle X_1\rangle \cap \langle X_2 \rangle$, where $\langle X_i \rangle$ denotes the span of $X_i$ in $\P^{n-1}$.
\end{Def}

\begin{Prop}
Suppose $G$ is the clique sum of two subgraphs $G_1$ and $G_2$ along a clique $K = G_1\cap G_2$. Then $X_G = X_{G_1}\cup X_{G_2}$ is linearly joined along the coordinate subspace $X_K = \langle \{e_i\colon i\in K\} \rangle$. \qed
\end{Prop}

We will show that the {\grank} is well-behaved under clique sums. We need the following Lemma in the proof of Theorem \ref{thm:rankcliquesum}.
\begin{lem}\label{lem:oneentry}
Let $M$ be an $n\times n$ symmetric matrix of full rank with complex entries. Let $a_0,a_1,\dots,a_n$, $b_1,b_2,\dots,b_{n+1}$ be complex numbers and consider the matrix
\[
P(t) = \begin{pmatrix}
a_0    & a_1 & \dots & a_n & t      \\
a_1    &     &       &     & b_1    \\
\vdots &     &   M   &     & \vdots \\
a_n    &     &       &     & b_n    \\
t      & b_1 & \dots & b_n & b_{n+1}
\end{pmatrix},
\]
which is symmetric of size $(n+2)\times (n+2)$.
\begin{compactenum}[(a)]
\item There exists a value $t_0\in\C$ such that the matrix $P(t_0)$ has rank $n+1$.
\item If the upper left $(n+1)\times (n+1)$ block has full rank $n+1$ and the bottom right $(n+1)\times (n+1)$ block has rank $n$, then there is a unique value $t_0\in\C$ such that $P(t_0)$ has rank $n$. This value is determined by the linear relation among the rows (or equivalently the columns) of the bottom right block.
\item If the upper left $(n+1)\times (n+1)$ as well as the bottom right $(n+1)\times (n+1)$ block have rank $n$ then there exists $t_0\in \C$ such that $P(t_0)$ has rank $n$. The entry $t_0$ is determined by the linear relations among the rows of the bottom right block.
\end{compactenum}
\end{lem}

\begin{proof}
The coefficient of $t^2$ in the determinant of $P(t)$ is the determinant of $M$, which is non-zero. Any zero of the determinant is a value $t_0$ satisfying (a).
For part (b), we use Schur's determinantal formula \cite[Theorem~1.1]{ZhangMR2160825}. Then we have
\begin{eqnarray*}
\det(P(t)) & = & \det(A) \det(b_{n+1} - (t,b_1,\dots,b_n)A^{-1} (t,b_1,\dots,b_n)^t) \\
& = & \det(A)\left(b_{n+1} - (t,b_1,\dots,b_n)A^{-1} (t,b_1,\dots,b_n)^t\right),
\end{eqnarray*}
where $A$ is the upper left $(n+1)\times (n+1)$ block of $P(t)$, which is invertible by assumption. We show that this polynomial in $t$ is a square. Essentially, this is because $b_{n+1}$ is equal to $\ul{b}^tM^{-1} \ul{b}$, which follows from Schur's determinantal formula applied to the bottom right $(n+1)\times (n+1)$ block, where $\ul{b} = (b_1,\dots,b_n)^t$. We have
\[
A^{-1} = 
\begin{pmatrix}
1 & 0 \\
-M^{-1} \ul{a} & I_n\\
\end{pmatrix}
\begin{pmatrix}
a_0 - \ul{a}^t M^{-1}\ul{a} & 0 \\
0 & M^{-1}
\end{pmatrix}
\begin{pmatrix}
1 & - \ul{a}^t M^{-1} \\
0 & I_n
\end{pmatrix}
\]
by \cite[Theorem~1.2]{ZhangMR2160825} (see the proof). Combining these results using the symmetry of $P(t)$ shows that the determinant is indeed a square, namely
\[
\det(P(t)) = \left(a_0 - \ul{a}^t M^{-1}\ul{a}\right) \left(t-\ul{b}^tM^{-1}\ul{a}\right)^2,
\]
where $\ul{a} = (a_1,\dots,a_n)^t$. 
The zero of the determinant is the entry determined by the linear relations among the columns of the bottom right $(n+1)\times (n+1)$ block.
To prove part (c), write $v_1,\dots,v_n$ for the rows of the matrix $M$. Since the upper left $(n+1)\times (n+1)$ block of $P$ has rank $n$, there are $u_1,\dots,u_n$ such that $(a_0,a_1,\dots,a_n) = u_1(a_1,v_1) + u_2(a_2,v_2) + \dots + u_n(a_n,v_n)$. Arguing similarly for the bottom right $(n+1) \times (n+1)$ block, we find $u_1',u_2',\dots,u_n'$ such that $(b_1,\dots,b_n,b_{n+1}) = u_1'(v_1,b_1) + \dots + u_n'(v_n,b_n)$. Setting the missing entry to be $u_1b_1 + u_2b_2 + \dots + u_n b_n$ proves the claim. Indeed, it is enough to show that $u_1b_1 + u_2b_2 + \dots + u_n b_n = u_1' a_1 + u_2' b_2 + \dots + u_n' b_n$. This follows from symmetry of $M$ because
\begin{eqnarray*}
& (u_1,\dots,u_n,0) (b_1,\dots,b_{n+1})^t = (u_1,\dots,u_n,0) \left( u_1'(v_1,b_1) + \dots + u_n' (v_n,b_n)\right)^t \\
& = (0,u_1',\dots,u_n') \left(u_1 (a_1,v_1) + \dots + u_n (a_n,v_n)\right)^t = (0,u_1',\dots,u_n') (a_0,a_1,\dots,a_n)^t.
\end{eqnarray*}
\end{proof}

\begin{thm}\label{thm:rankcliquesum}
Let $G$ be the clique sum of two subgraphs $G_1$ and $G_2$. Then the {\grank} of $G$ is equal to the maximum of the {\grank s} of $G_1$ and $G_2$.
\end{thm}

\begin{proof}
Let $G$ be the clique sum of $G_1$ and $G_2$ along a clique $K\subset G$ of size $k$. We prove the claim by taking the geometric point of view. The {\grank} of $G$ is the smallest $r$ such that a generic matrix in $\C[X_G]_2 = (\C[x_1,\dots,x_n]/I_G)_2$ can be represented by a quadratic form in $\C[x_1,\ldots,x_n]_2$ of rank $r$. After relabeling the vertices of $G$, if necessary, we can assume that an extension of a quadratic form in $\C[X_G]_2$ is a block matrix

\begin{center}
\begin{tikzpicture}
\matrix (M) [matrix of math nodes,left delimiter={(},right
delimiter={)}]
{A_1 &   & ?    \\
     & C &      \\
?    &   & A_2  \\
};
\draw[black,thick] (M-1-1.north west) rectangle (M-2-2.south east);
\draw[black,thick] (M-2-2.north west) rectangle (M-3-3.south east);
\end{tikzpicture}
\end{center}
where the upper left block is the restriction to $X_{G_1}$ and the lower right block is the restriction to $X_{G_2}$. Note that they have a common $k\times k$ submatrix, denoted $C$, which is the restriction to $X_{G_1}\cap X_{G_2} = \scp{e_i\colon i\in K}$. The entries in the top right and lower left are undetermined in $\C[X_G]_2$ and choices of all unknown unspecified entries in this matrix give an extension to $\C[x_1,\dots,x_n]_2$.

If the entire matrix has a completion of rank $r$, then so do the top left block and bottom right block. This shows that the {\grank} of $G$ is at least the maximum of the {\grank s} of $G_1$ and $G_2$.

For the reverse inequality, let the {\grank s} of $G_1$ and $G_2$ be $r_1$ and $r_2$, respectively, and assume $r_1\geq r_2$. We will inductively construct a completion of rank $r_1$ using the preceding Lemma~\ref{lem:oneentry}. First, given generic entries, we can complete the blocks $A_1$ and $A_2$ to matrices of rank $r_1$ and $r_2$, respectively. Further, we can assume by genericity that the bottom right $r_1\times r_1$ block of $A_1$ has full rank and the top left $r_2\times r_2$ block of $A_2$ has full rank. If $r_2>k = \rk(C)$, we grow $A_2$ by one row in the following way (compare Figure~\ref{fig:comp}): Add generic entries (for the stars) in the new row until we are at the $r_2$-th column of $A_2$ (the triangle). Choose this entry such that the determinant of the $(r_2+1)\times (r_2+1)$ block formed with entries from $A_1$ in the new row is $0$ (Lemma~\ref{lem:oneentry}(a)), i.e.~the last column of this block is a linear combination of the first $r_2$ columns. By Lemma~\ref{lem:oneentry}(b), we can fill in the entire new row by the representation of every column of $A_2$ as a linear combination of the first $r_2$. In this way, we get an extension $A_2'$ of $A_2$ by one row such that the first $r_2$ columns of $A_2'$ are a basis of the column span of $A_2'$ and the top left $r_2\times r_2$ block has full rank. This process enlarges the common block $C$ of $A_1$ and $A_2$ and grows its rank by $1$. So we can inductively add rows to $A_2$ until $r_2 = k$. 

\begin{figure}[h]
\includegraphics[scale = 0.6]{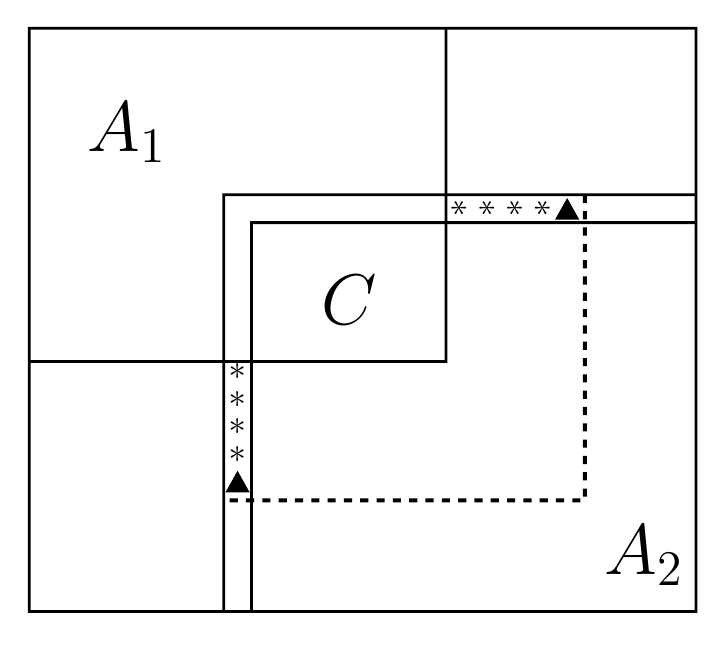}
\caption{Visualization of the growing step in the proof of Theorem~\ref{thm:rankcliquesum}.}
\label{fig:comp}
\end{figure}

If $r_1>r_2 = k$, we add one row to $A_2$ similarly as before: Fill in each entry of the row according to the linear relation of the columns by Lemma~\ref{lem:oneentry}(b). The obtained extension $A_2'$ of $A_2$ will have rank $r_2+1$, because the first $r_2+1$ columns of $A_2'$ are now a basis of the column span of $A_2'$.

So we have now reduced to the case that $r_1 = r_2 = k$. Now we fill in the missing entries in the top right block by the linear relations among the rows of $A_1$, which gives the same as using the linear relations among the rows of $A_2$ in the bottom left block by Lemma~\ref{lem:oneentry}(c). In the end, we get a completion of rank $r_1$, which proves the claim.
\end{proof}

We now examine behavior of the \mlt under clique sums. We begin with a Lemma.

\begin{lem}\label{lemma:dualcones}
Let $G$ be the clique sum of $G_1$ and $G_2$. Then a quadratic form $f\in\R[X_G]_2$ lies in the interior of $\Sigma_G$ if and only if its restrictions to $X_{G_1}$ and $X_{G_2}$ lie in the interior of $\Sigma_{G_1}$ and $\Sigma_{G_2}$, respectively.
\end{lem}

\begin{proof}
Since $G$ is a clique sum of graphs, $X_G = X_{G_1}\cup X_{G_2}$ is linearly joined. By \cite[Corollary~5.4]{BlekhermanSinnVelasco}, we have $\Sigma_G^\vee = \conv(\Sigma^\vee_{G_1}\cup \Sigma^\vee_{G_2})$, where both cones on the right hand side are faces of $\Sigma_G^\vee$. In particular, the union of extreme rays of $\Sigma_G^\vee$ is the union of the extreme rays of $\Sigma_{G_1}^\vee$ and $\Sigma_{G_2}^\vee$. By the separation theorem in convex geometry, a point $x$ is in the interior of $\Sigma_G$ if and only if for every extreme ray $\R_+ \ell$ of $\Sigma_G^\vee$ we have $\ell(x)>0$. This proves the claim.
\end{proof}

\begin{Thm}\label{thm:mltcliquesum}
Let $G$ be the clique sum of $G_1$ and $G_2$. The {\mlt} of $G$ is the maximum of the \mlt s of $G_1$ and $G_2$.
\end{Thm}

\begin{proof}
We assume, after relabeling of the vertices, if necessary, that $G_1 = ([p],E_1)$ and $G_2 = ([q]+r, E_2)$, so that $\dim \langle X_{G_1} \rangle = p-1$ and $\dim \langle X_{G_2} \rangle = q-1$. Then the composition of the projection $\pi_G\colon \sym[n] \to \R[X_G]_2$ with the restriction map ${\rm res}_1\colon \R[X_G]_2 \to \R[X_{G_1}]_2$ is the projection $\pi_1\colon \sym[p] \to \R[X_{G_1}]_2$ of the upper left $p\times p$ block of the $n\times n$ matrix. 
The analogous statement for the lower right $q\times q$ block holds for $\pi_2\colon \sym[q] \to \R[X_{G_2}]_2$. 

First, we show $\ml(G)\leq \max\{\ml(G_1),\ml(G_2)\}$. Let $r = \max\{\ml(G_1),\ml(G_2)\}$ and let $A$ be a generic positive semidefinite rank $r$ matrix in $\sym[n]$. Then the upper left $p\times p$ block $A_1$ and the lower right $q\times q$ block $A_2$ of $A$ are generic matrices of rank $\min\{r,p\}$ and $\min\{r,q\}$, respectively, with the property that the restriction of $\pi(A)$ to $X_{G_1}$ and $X_{G_2}$ are the projections $\pi_1(A_1)$ and $\pi_2(A_2)$ of the blocks. By definition of the {\mlt}, we know that $\pi_1(A_1)$ is an interior point of $\Sigma_{G_1}$ and $\pi_2(A_2)$ is an interior point of $\Sigma_{G_1}$. So $\pi(A)$ is an interior point of $\Sigma_G$ by Lemma~\ref{lemma:dualcones}, which shows the desired inequality.

Now, we show the inequality $\ml(G_1)\leq \ml(G)$. Suppose $r+1 = \ml(G_1)$. Then there exists a positive semidefinite $p\times p$ matrix $A_1$ and a semialgebraic neighborhood $U$ of $A_1$ inside the variety of $p\times p$ matrices of rank at most $r$ such that every matrix $A\in U$ maps to the boundary of $\Sigma_{G_1}$ under $\pi_1$. The set $T$ of all positive semidefinite $n\times n$ matrices of rank $r$, whose upper left block of size $p\times p$ is a matrix in $U$ is a semialgebraic subset of the variety $\V_r^{n}$ of $n\times n$  matrices of rank $r$ that has nonempty interior.
So take $B\in T$. Then the upper left $p\times p$ block of $B$ maps to the boundary of $\Sigma_{G_1}$, which can be certified by an extreme ray of $\Sigma_{G_1}^\vee$, i.e.~there exists an extreme ray $\R_+\ell$ of $\Sigma_{G_1}^\vee$ such that $\ell(B) = 0$. 
Now the fact that this extreme ray is also an extreme ray of $\Sigma_G^\vee$, see Lemma~\ref{lemma:dualcones}, certifies that $\pi_G(B)$ is a boundary point of $\Sigma_G$ for every $B\in T$. This shows the reverse inequality.
\end{proof}

\begin{Cor}\label{cor:wheels}
The {\grank} and {\mlt} are equal for any graph which does not contain a wheel $W_k$ for $k\geq 5$ or a splitting of a wheel $W_k$ for $k\geq 4$ as an induced subgraph.
\end{Cor}

The wheel $W_k$ is a graph on $k$ vertices which is composed of a cycle $C_{k-1}$ on $k-1$ nodes together with an additional vertex adjacent to every vertex of $C_{k-1}$.

\begin{proof}
By a result of Johnson and McKee \cite{JohnsonMcKeeMR1415290}, these graphs are clique sums of chordal graphs and series-parallel graphs. 
Moreover, results of \cite{BuhlMR1241392} and \cite{UhlerMR3014306} imply that {\grank} and {\mlt} also agree on series-parallel graphs, which are graphs of treewidth at most $2$. Therefore, for clique sums of chordal and series-parallel graphs, {\grank} and {\mlt} are equal by Theorems~\ref{thm:rankcliquesum} and \ref{thm:mltcliquesum}.
\end{proof}

We will see in Section~\ref{sec:bipartite} that {\grank} and {\mlt} are not equal in general. 

The sparsity order of a graph $G$ is the largest rank of an extreme ray of $\Sigma_{G}^\vee$, see \cite{AHMRMR960140} or \cite{LaurentSparsityOrder}. The positive semidefinite matrix completion theorem \cite[Theorem~7]{PSDComp} implies that the sparsity order of $G$ is $1$ if and only if $G$ is chordal. We now prove an inequality which relates the sparsity order and the generic completion rank of $G$.

\begin{Prop}
Let $G$ be a graph on $n$ vertices. Then the sum of the {\grank} of $G$ and its sparsity order is at most $n+1$.
\end{Prop}

\begin{proof}
Let $k$ be the sparsity order of $G$. Then the kernel of an extreme ray of rank $k$ is an $(n-k)$-dimensional linear series in $R_1$. A result of Blekherman \cite[Proposition~4.2]{BlMR3272733} shows that such a linear series will generate a hyperplane in $R_2$. This means that a generic linear series of dimension $n-k$ generates at least a hyperplane in $R_2$, so a generic linear series of dimension $n-k+1$ will certainly generate all of $R_2$. This implies that the {\grank} is at most $n-k+1$ by Proposition~\ref{prop:rankdefs}(1), which proves the claimed inequality.
\end{proof}

This estimate is sharp in some cases. 
\begin{Exm}
Consider a complete bipartite graph $K_{m,n}$ on $m+n$ vertices. Assume $m\leq n$. We will see in Section~\ref{sec:bipartite} that the {\grank} of $K_{m,n}$ is $m$ as long as $n\leq \binom{m}{2}$ and $m+1$ otherwise. The sparsity order of complete bipartite graphs is determined in \cite[Theorem~2.1]{GrPiMR1057063}: it is $n$ as long as $n\leq \binom{m}{2}+1$ and $\binom{m}{2}+1$ otherwise. So the sum of the {\grank} and the sparsity order is always at most $m+n$ and it is equal to $m+n$ only if $n= \binom{m}{2}+1$.
\end{Exm}

\section{Complete Bipartite Graphs}\label{sec:bipartite}
Let $n\geq m \geq 2$ be positive integers and denote by $K_{m,n}$ the complete bipartite graph on two color classes of size $m$ and $n$, respectively. We assume $m\geq 2$ so that $K_{m,n}$ is not a tree. In this section, we will first determine the {\grank} of $K_{m,n}$ and then its {\mlt}.

\begin{Thm}\label{thm:rankbipartite}
Let $n\geq m \geq 2$. The {\grank} of $K_{m,n}$ is $m+1$ if $n>\binom{m}{2}$. It is $m$ if $n\leq \binom{m}{2}$.
\end{Thm}

We will give a proof of this theorem below after a few preparations. The following remark is a well-known fact from graph theory.

\begin{Rem}\label{rem:treewidth}
If $m\geq 2$, the treewidth of $K_{m,n}$ is $m+1$.
\end{Rem}

\begin{Rem}\label{rem:Kmnrankbound}
The {\grank} of $K_{m,n}$ is at least $m$, simply because a $K_{m,n}$-partial matrix is of the form
\[
\begin{pmatrix}
D_1 & B \\
B^t & D_2
\end{pmatrix},
\]
where $B$ is an $m\times n$ matrix and $D_1$ is an $m\times m$ matrix, $D_2$ is an $n\times n$ matrix, both of which are only specified along their diagonals. So generically, a completion of a $K_{m,n}$-partial matrix will have rank at least $m$, which is the generic rank of $B$.
\end{Rem}

\begin{prop}\label{prop:Kmnlargen}
Suppose $n>\binom{m}{2}$. Then the {\grank} of $K_{m,n}$ is at least $m+1$.
\end{prop}

\begin{proof}
The {\grank} of $K_{m,n}$ is at least $m$ by the previous Remark~\ref{rem:Kmnrankbound}. It cannot be $m$ by dimension count. Indeed, the dimension of the variety $V_m^{m+n}$ of symmetric matrices of size $m+n$ and rank at most $m$ is $\binom{m+1}{2} + m\cdot n = \binom{m}{2} + m + m\cdot n$. Since $n> \binom{m}{2}$, the dimension $m+n+m\cdot n$ of the target of the projection $\pi_{K_{m,n}}\colon \sym[m+n] \to \R^{m+n} \oplus \R^{\# E}$ is larger than the dimension of $V_m^{m+n}$. Therefore, the restriction of the projection $\pi_{K_{m,n}}$ to $V_m^{m+n}$ cannot be dominant.
\end{proof}

\begin{thm}\label{thm:rankm}
Suppose $m\geq 2$ and $n = \binomial{m}{2}$. Then the {\grank} of $K_{m,n}$ is equal to $m$.
\end{thm}

\begin{proof}
The {\grank} of $K_{m,n}$ is at least $m$. We show by example that a generic linear series of dimension $m$ will generate $R_2$, which implies that the {\grank} is equal to $m$ by Proposition~\ref{prop:rankdefs}(1).

Suppose not, then for every linear series $W\subset R_1$ of dimension $m$ there exists a matrix $M\in L_{K_{m,n}} = \ker(\pi_{K_{m,n}})^\perp \cong R_2^\ast$ such that $W\subset \ker(M)$. Write
\[
M = 
\begin{pmatrix}
D_m & B \\
B^t & D_n
\end{pmatrix},
\]
where $D_m$ and $D_n$ are diagonal matrices of size $m\times m$, resp.~$n\times n$.
By choosing an appropriate basis of $W$, we can assume that $W$ is the row span of $(I_m \quad A^t)$, where $A$ is an $n\times m$ matrix. The condition $W\subset \ker(M)$ is equivalent to
\[
\begin{pmatrix}
D_m + BA \\
B^t + D_n A
\end{pmatrix} = 
\begin{pmatrix}
D_m & B \\
B^t & D_n
\end{pmatrix}
\begin{pmatrix}
I_m \\
A
\end{pmatrix} = 0.
\]
This system of equations implies $D_m - A^tD_n A = 0$ by solving the second row for $B^t$ and plugging $B^t = -D_n A$ into the first equation. Given $A$, this is a system of $m+\binom{m}{2}$ linear equations in the $m+n$ variables, which are the diagonal entries of $D_m$ and $D_n$. Since the equations are in row-echelon form for the entries of $D_m$, we focus on the off-diagonal equations in the matrix $D_m - A^t D_n A = 0$ as equations in the diagonal entries of $D_n$. Write $D_n = \diag(\lambda_1,\lambda_2,\dots,\lambda_n)$ and $A = (a_{ij})_{1\leq i \leq n, 1\leq j \leq m}$. Then these equations are $\sum_{k=1}^n \lambda_k a_{ki} a_{kj} = 0$. The rows of the coefficient matrix $C$ of this linear system in the variables $\lambda_1,\lambda_2,\dots,\lambda_n$ are indexed by pairs $(i,j)$ for $1\leq i < j \leq m$ and the $(i,j)$th row is equal to $(a_{1i}a_{1j}, a_{2i} a_{2j}, \dots, a_{ni} a_{nj})$. We choose $A$ in such a way that the determinant of the coefficient matrix $C$ is non-zero. Then the only solution to this system is $D_n = 0$, which implies $D_m = 0$ and $B^t = 0$ by the above matrix equations. So $M = 0$, which implies that this linear series $W$ generates $R_2$.

We fix an order of the $n = \binom{m}{2}$ pairs $(i,j)$, where $1\leq i<j\leq m$. We set the $k$th row of $A$ to be $e_i+e_j$, where $(i,j)$ is the $k$th pair in our ordering and $e_i$ and $e_j$ are the $i$th, resp.~$j$th standard basis vectors of $\R^m$. This implies that there is exactly one entry $1$ in every row and column of the coefficient matrix $C$. Indeed, in the row of $C$ indexed by the pair $(i,j)$, only the entry $a_{ki} a_{kj}$ is non-zero and equal to $1$, where again $(i,j)$ is the $k$th pair in the ordering. Similarly, in the $k$th column of $C$, only the entry $a_{ki} a_{kj}$ is non-zero and equal to $1$. This shows that $C$ is a permutation matrix and therefore has full rank.
\end{proof}

\begin{proof}[Proof of Theorem~\ref{thm:rankbipartite}]
The {\grank} of $K_{m,n}$ is equal to $m$ in case $n = \binomial{m}{2}$ by Proposition~\ref{thm:rankm}. If $m\leq n<\binomial{m}{2}$, the {\grank} only decreases because we are deleting edges. On the other hand, it cannot drop below $m$ by Remark~\ref{rem:Kmnrankbound}, so it must be equal to $m$. If $n>\binomial{m}{2}$, then the {\grank} is larger than $m$ by Proposition~\ref{prop:Kmnlargen}. Since the treewidth of $K_{m,n}$ is $m+1$ by Remark~\ref{rem:treewidth}, it cannot be larger than $m+1$ by \cite[Proposition~1.3]{GrossSullivant}, which finishes the proof.
\end{proof}

We now turn to the {\mlt} of complete bipartite graphs. It turns out that it is different from the {\grank} for certain values of $n\geq m\geq 5$.

Below, we use the Grassmannian for a dimension count. Here is a short summary of affine coordinates on it, see e.g.~\cite{HarMR1416564}.
\begin{Rem}\label{rem:grassmannian}
Given a $k$-dimensional linear space $L\subset\R^m$, we choose a basis $v_1,\dots,v_k$ and put these vectors as the rows of a matrix. After row operations and possibly permutation of the columns, we can write this matrix as
$\begin{pmatrix}
I_k & A
\end{pmatrix}$,
where $A$ is a $k\times (m-k)$ matrix. The entries of $A$ are affine coordinates of the linear space $L$ on an affine chart of the Grassmannian $\gr(k,m)$ of $k$-dimensional linear subspaces of $\R^m$.
\end{Rem}

\begin{Thm}\label{thm:bipartite}
Let $n\geq m\geq 2$ and set $M$ to be the smallest $k$ such that $\binomial{k+1}{2}\geq m+n$. The {\mlt} of the complete bipartite graph $K_{m,n}$ is the minimum between $M$ and $m+1$.
\end{Thm}
Note that the {\grank} of $K_{m,n}$ is $m$ or $m+1$, depending on the size of the larger part $n$ compared to $m$, and therefore, this is an upper bound on the {\mlt}. The above minimum is $M$ for $n\leq \binomial{m}{2}$ and $m+1$ for $n>\binomial{m}{2}$.

We prove the theorem by two dimension counts, which we state separately. The goal is to provide a lower bound and a matching upper bound using Proposition~\ref{prop:mltdefs}(2).
We write $\Sigma_{m,n} = \pi_{K_{m,n}}(\psd[m+n])$ for the cone of sums of squares of linear forms on the line arrangement $X_{K_{m,n}}$ in $\P^{m+n-1}$ defined by $K_{m,n}$ and $\Sigma_{m,n}^\vee$ for its dual cone. 
For matrices in the orthogonal complement $L_{K_{m,n}}$ of $\ker(\pi_{K_{m,n}})$, we fix the notation
\[
M = 
\begin{pmatrix}
D_m & A \\
A^t & D_n
\end{pmatrix},
\]
where $D_m$ is an $m\times m$ diagonal matrix, $D_n$ is an $n\times n$ diagonal matrix, and $A$ is an $m\times n$ matrix.

\begin{lem}\label{lem:mltlowerbound}
Let $n\geq m\geq 2$ and set $M$ to be the smallest $k$ such that $\binomial{k+1}{2}\geq m+n$. The {\mlt} of the complete bipartite graph $K_{m,n}$ is at least $\min\{M,m+1\}$.
\end{lem}

\begin{proof}
Since every extreme ray of $\Sigma_{m,n}^\vee$ has rank at most $n$ by \cite[Theorem~2.1]{GrPiMR1057063}, every linear series that is in the kernel of a matrix $M\in\Sigma_{m,n}^\vee$ is in the kernel of a matrix $M'\in\Sigma_{m,n}^\vee$ of rank at most $n$.
We show that we can assume that $M$ is positive semidefinite of rank $n$ and the diagonal entries of $D_m$ and $D_n$ are nonzero.
If some of the diagonal entries of $D_m$ or $D_n$ are equal to $0$, then $M$ being positive semidefinite implies that the corresponding rows and columns of $M$ have to be $0$. In terms of the graph, this corresponds to the deletion of the vertices corresponding to these zero rows and columns. So this reduces to the case of $K_{m',n'}$, where $m'\leq m$ and $n'\leq n$. So we can assume that the diagonal entries of $D_m$ and $D_n$ are non-zero, which implies that the rank of $M$ is $n$.

With these assumptions, we conclude from the rank additivity of the Schur complement (\cite[Section~0.9]{ZhangMR2160825}) that $D_m - A D_n^{-1} A^t = 0$. A vector $(v,w)^t$, where $v\in \R^m$ and $w\in \R^n$, is in the kernel of $M$ if and only if $w = -D_n^{-1}A^t v$ by the bottom blocks of $M$. It is convenient to interpret this equation differently. Set $B = D_n^{-1/2} A^t D_m^{-1/2}$, where $D_m^{-1/2}$ and $D_n^{-1/2}$ denote the inverse of a square root of $D_m$ and $D_n$, respectively. Then $D_n^{1/2} w = - B D_m^{1/2} v$ and the above equation $D_m - AD_n^{-1}A^t = 0$ shows that $B$ is an isometric embedding of $\R^m$ into $\R^n$, i.e.~$B^tB = I_m$.

We show that for $k\leq m$, there is an open ball $B\subset\gr(k,m+n)$ in the Euclidean topology such that every $k$-dimensional subspace $W\in B$ with basis $(v_1,w_1)^t,\dots,(v_k,w_k)^t$, such that the matrix $V = (v_1,\dots,v_k)^t$ has rank $k$, is contained in the kernel of a matrix $M\in\Sigma_{m,n}^\vee$ whenever $\binomial{k+1}{2}<m+n$.
To prove this claim, we find an isometric embedding $B\colon \R^m \to \R^n$ and diagonal matrices $D_m$ and $D_n$ such that $D_n w_i = BD_m v_i$. Such an isometric embedding exists if and only if $\scp{D_m v_i,D_m v_j} = \scp{D_n w_i, D_n w_j}$ for all $i,j = 1,\dots,k$. 

The vectors $v_i$ and $w_j$ are given by the basis of $W\subset \R^{m+n}$. We read these equations as a homogeneous linear system whose variables are squares of the $m+n$ entries of the diagonal matrices $D_m$ and $D_n$. So we have $\binomial{k+1}{2}$ conditions in $m+n$ variables. There exists a nontrivial solution if $\binomial{k+1}{2} < m+n$. Therefore, $W$ is contained in the kernel of the matrix
\[
M = 
\begin{pmatrix}
D_m^2 & -D_m B^t D_n \\
-D_n B D_m & D_n^2
\end{pmatrix}
\]
by the above computation using the Schur complement. By the properties of the Schur complement, see \cite[Theorem~1.12]{ZhangMR2160825}, and the above argument for the existence of $B$, $M$ is positive semidefinite if $D_m$ and $D_n$ have real entries. The semialgebraic set of linear spaces in $\gr(k,m+n)$ such that $D_m$ and $D_n$ have real entries has non-empty interior, which we will prove in the following Lemma~\ref{lem:mltrealsolutions}.

This argument shows that for $k\leq m$, the semialgebraic set of $k$-dimensional subspaces of $\R^{m+n}$ in $\gr(k,m+n)$ that are contained in the kernel of positive semidefinite matrices $M$ of rank $n$ has non-empty interior whenever $\binomial{k+1}{2}<m+n$. In particular, if $n>\binomial{m}{2}$, we find an open ball in $\gr(m,m+n)$ of subspaces that are in the kernel of positive semidefinite matrices.
So by Proposition~\ref{prop:mltdefs}, the {\mlt} of $K_{m,n}$ is at least $\min\{k\colon \binomial{k+1}{2}\geq m+n\}$ for $n\leq \binomial{m}{2}$ and at least $m+1$ for $n>\binomial{m}{2}$. 
\end{proof}

\begin{lem}\label{lem:mltrealsolutions}
Let $k\leq m\leq n$ be integers and assume $\binom{k+1}{2}<m+n$. Then there exists an open ball $B\subset \gr(k,m+n)(\R)$ in the Euclidean topology consisting of real subspaces with basis $(v_1,w_1), (v_2,w_2), \ldots, (v_k,w_k)$, $v_i\in \R^m$ and $w_i\in\R^n$, such that the equations $\scp{D_mv_i,D_mv_j} = \scp{D_nw_i,D_nw_j}$ for all $1\leq i\leq j\leq k$, where $D_m$ and $D_n$ are diagonal matrices, have a real solution $D_m$ and $D_n$ of full rank.
\end{lem}

\begin{proof}
As in the proof of the preceding Lemma~\ref{lem:mltlowerbound}, we think of the equations $\scp{D_m v_i,D_mv_j} - \scp{D_nw_i,D_nw_j} = 0$ as linear equations in the squares of the diagonal entries of $D_m$ and $D_n$. Let $C$ be the coefficient matrix of this linear system. Its rows are indexed by pairs $(i,j)$ with $1\leq i\leq j\leq k$ specifying the equation and its columns are indexed by $a = 1,\dots,m$, the components of the vectors $v_i$ and then $b = 1,\dots,n$, the components of the vectors $w_i$. The $(i,j),a$ entry in the left $\binom{k+1}{2}\times m$ block is given by $v_{i,a}v_{j,a}$ and the $(i,j),b$ entry in the right $\binom{k+1}{2}\times n$ block is given by $-w_{i,b}w_{j,b}$. 

Generically, the kernel of $C$ will contain a vector with only non-zero entries. Indeed, if every vector in the kernel of $C$ had a zero entry, there would be a position such that all vectors in the kernel have a zero there. For the matrix $C$, that means that the corresponding column is not contained in the span of the other columns. By symmetry of the columns of $C$, there would then be a non-empty open set of vectors $(v_1,w_1),\ldots,(v_k,w_k)$ such that the same holds for each column. Since $(\A^{m+n})^k$ is irreducible, the intersection of these non-empty open sets  would be non-empty. This would mean that $C$ has rank $m+n$, which is absurd because $\binom{k+1}{2}<m+n$.

In other words, the following map is dominant
\[
\gamma\colon \left\{
\begin{array}[]{cc}
(\C^\ast)^m\times (\C^\ast)^n\times Y \to \gr(k,m+n)\\
(D_m,D_n,(v_1,\ldots,v_k,w_1,\ldots,w_k))\mapsto 
\begin{pmatrix}
\left(D_m^{-1}v_1\right)^t & \left(D_n^{-1}w_1\right)^t\\
\vdots & \vdots \\
\left(D_m^{-1}v_k\right)^t & \left(D_n^{-1}w_k\right)^t
\end{pmatrix}
\end{array}\right.,
\]
where $Y$ is the variety of vectors $(v_1,\ldots,v_k,w_1,\ldots,w_k)$ in $(\A^m)^k\times (\A^n)^k$ such that $\scp{v_i,v_j} = \scp{w_i,w_j}$ and $v_1,\ldots,v_k$ are linearly independent. For every irreducible component of $Y$, its real points are Zariski-dense. So, by Generic Smoothness \cite[Corollary III.10.7 and Proposition III.10.4]{HarMR0463157}, the image of this map restricted to real matrices $D_m$ and $D_n$ and real vectors $(v_1,\ldots,v_k,w_1,\ldots,w_k)$ has full dimension in the whole image of $\gamma$ as a semi-algebraic set. This is a reformulation of the claim.
\end{proof}

\begin{lem}\label{lem:mltupperbound}
Let $n\geq m\geq 2$ and set $M$ to be the smallest $k$ such that $\binomial{k+1}{2}\geq m+n$. The {\mlt} of the complete bipartite graph $K_{m,n}$ is at most $\min\{M,m+1\}$.
\end{lem}

\begin{proof}
Since the {\grank} of $K_{m,n}$ is $m+1$ in the case $n>\binomial{m}{2}$ by Theorem~\ref{thm:rankbipartite}, the {\mlt} of $K_{m,n}$ is at most $m+1$ in this case.
So from now on, we assume $n\leq \binomial{m}{2}$. In this case, the integer $M$ is at most $m$. We estimate the dimension of the image of the map
\[
\phi\colon\left\{
\begin{array}[]{ccc}
\R^n \times O \times \R^m \times \gr(k,m) & \to & \gr(k,m+n) \\
(D_n,B,D_m,V) & \mapsto & \left(V | VD_mB^tD_n^{-1} \right)
\end{array},
\right.
\]
where $O$ is the set of $n\times m$ matrices such that $B^t B = I_m$ and $M\leq k \leq  m$. Again, the computation using the Schur complements at the beginning of the proof of Lemma~\ref{lem:mltlowerbound} shows that the linear spaces in the image of this map are the kernels of matrices of the form
\[
M =
\begin{pmatrix}
D_m^2 & -D_m B^t D_n \\
-D_n B D_m & D_n^2
\end{pmatrix}.
\]
The dimension of the domain is $n + \sum_{i=1}^m(n-i) + m + k(m-k)$ and the dimension of the image is at most $n+\sum_{i=1}^k(n-i) + m + k(m-k)-1$ because $B$ is only determined on a $k$-dimensional subspace of $\R^m$ and there is a $1$-dimensional stabilizer. Indeed, given $\alpha\in\R^\ast$, we have $\phi(\diag_n(\alpha)D_n,B,D_m\diag_m(\alpha),V) = \phi(D_n,B,D_m,V)$. We compare it to the dimension of $\gr(k,m+n)$:
\[
k(m+n-k) - \left( n+\sum_{i=1}^k(n-i) + m + k(m-k) - 1 \right) = -m-n + \binomial{k+1}{2} + 1.
\]
So, if $k\leq m$ and $\binomial{k+1}{2} > m+n-1$, then the image of $\phi$ cannot have non-empty interior in $\gr(k,m+n)$. In other words, the {\mlt} of $K_{m,n}$ is at most the smallest $k$ such that $\binomial{k+1}{2} \geq m+n$ by Proposition~\ref{prop:mltdefs}(2). 
\end{proof}

\begin{proof}[Proof of Theorem~\ref{thm:bipartite}]
The lower bound in Lemma~\ref{lem:mltlowerbound} matches the upper bound in Lemma~\ref{lem:mltupperbound} on the {\mlt} of $K_{m,n}$. So we conclude that the {\mlt} of $K_{m,n}$ is the smallest $k$ such that $\binomial{k+1}{2} \geq m+n$ if $n\leq \binomial{m}{2}$, and $m+1$ if $n>\binomial{m}{2}$. 
\end{proof}

\begin{Cor}\label{cor:gcrvsmlt}
\begin{compactenum}[(a)]
\item The {\mlt} of $K_{5,5}$ is $4$, whereas the {\grank} is $5$.
\item Asymptotically, the {\mlt} of $K_{m,m}$ is of the order $2\sqrt{m}$. 
\item The {\mlt} of $K_{m,n}$ is equal to the {\grank} of $K_{m,n}$ for $m=2,3,4$ and whenever $n > \binom{m}{2}$.  
\end{compactenum}
\end{Cor}
The case $m=2$ in (c) was known before, see \cite[Proposition~4.2]{UhlerMR3014306}

\begin{proof}
(a) follows directly from Theorem~\ref{thm:bipartite}. (b) also follows from Theorem~\ref{thm:bipartite} because $\binomial{k+1}{2}\geq 2m$ is quadratic in $k$ and the constant coefficient, after rescaling to make the inequality monic in $k$, is $4m$. To prove (c), we discuss the several cases. For $m=2$, we are always in the case $n>\binomial{m}{2}$ because $n\geq m$, so the {\grank} of $K_{2,n}$ is $3$ and the {\mlt} is $3$. For $m=3$, we have two cases. If $n = 3 = \binomial{3}{2}$, the {\grank} is $3$ as well as the {\mlt}. If $n>3$, the {\grank} and {\mlt} are $4$. For $m=4$, the {\grank} and the {\mlt} are $4$ for $n = 4, 5, 6$ and both are $5$ for $n>6$.
\end{proof}

\begin{Rem}
It is essential for the dimension count in the proof of Theorem~\ref{thm:bipartite} that we are considering positive semidefinite matrices $M\in L_{K_{m,n}}$. For any linear series $W\subset R_1$ of dimension $m-1$, we can find a matrix $M$ of rank $2$ such that $W$ is contained in the kernel of $M$. Namely, choose a linear form $\scp{v,-}$ on $\R^m$ that contains the projection of $W$ onto the first $m$ components and a linear form $\scp{w,-}$ on $\R^n$ that contains the projection of $W$ onto the last $n$ components. Then the matrix
\[
M = 
\begin{pmatrix}
0 & v^tw \\
w^tv & 0
\end{pmatrix}
\]
has rank $2$ and contains $W$ in its kernel. Since the diagonal is $0$, it cannot be positive semidefinite.
\end{Rem}

\begin{Rem}
The {\mlt} of the complete bipartite graph $K_{m,m}$ is roughly $2\sqrt{m}$ and its {\grank} is $m$. By dropping edges, we can decrease the {\grank}. Using Theorem~\ref{thm:dimcount}, we know that a subgraph of $K_{m,m}$ of rank roughly $2\sqrt{m}$ can have at most about $4m\sqrt{m}$ edges. So in order for the {\mlt} and the {\grank} of $G$ to have a chance to be equal, $G$ needs to be relatively sparse.
\end{Rem}

\subsection{Non-symmetric Completion}\label{sec:non-sym}
We now relate the general low rank matrix completion problem over $\C$ to the matrix completion problem in the symmetric setup for the special case of bipartite graphs. 

Let $A$ be a partially specified $m\times n$ matrix. To $A$, we associate a bipartite graph $G$ on the disjoint union $[m]\cup [n]$, where the two parts index the rows and columns, respectively. We draw the edge $\{i,j\}$ in $G$ if and only if the $(i,j)$th entry of $A$ is specified. The \emph{non-symmetric {\grank}} of $G$ is the smallest $r$ such that a generic partially specified $m\times n$ matrix $A$ with the pattern of known entries given by $G$ has a completion of rank $r$. We want to relate the non-symmetric {\grank} to the {\grank} of $G$, i.e.~the {\grank} in the symmetric setup considered before.
\begin{Prop}\label{prop:matrixcompletion}
Let $G$ be a bipartite graph.
\begin{compactenum}[(a)]
\item The non-symmetric {\grank} of $G$ is at most the symmetric {\grank} of $G$.
\item The symmetric {\grank} of a bipartite graph $G$ is at most the non-symmetric {\grank} of $G$ plus one.
\end{compactenum}
\end{Prop}

\begin{proof}
Given a $G$-partial $m\times n$ matrix $A$, we consider the $G$-partial symmetric block matrix
\[
M' = \begin{pmatrix}
D_m & A \\
A^t & D_n
\end{pmatrix},
\]
where $D_m$ and $D_n$ are matrices whose diagonal entries are specified. To prove (a), let $M$ be a completion of rank $r$ of the $G$-partial symmetric matrix $M'$. Then there are $(m+n)\times r$ matrices $U$ and $V$ such that $M = U V^t$. Since $A$ is a block of $M$, such a factorization of $M$ gives a factorization $A = U_1 V_2^t$ of $A$, where $U_1$ and $V_2$ are the appropriate $m\times r$, resp.~$n\times r$, blocks of $U$, resp.~$V$, which shows (a).

To prove (b), define the projection map $\pi_G^n\colon M_{m\times n} \to \A^{\# E}$ from the set of $m\times n$ matrices to the space indexed by the edges of $G$ mapping a matrix $(a_{ij})$ to $(a_{ij}\colon \{i,j\}\in E)$. Let $r$ be the non-symmetric generic completion rank of $G$ so that $\pi_G^n$ restricted to the variety of matrices of rank at most $r$ is dominant. This implies the inequality $r (m+n) - r^2 - \# E\geq 0$ by dimension count. It also implies that the symmetric projection map $\pi_G\colon \sym\to \A^{\# E}$ that takes a symmetric matrix $(a_{ij})$ to $(a_{ij}\colon \{i,j\}\in E)$ restricted to the variety of symmetric matrices of rank at most $r$ is dominant. Indeed, if a generic $G$-partial $m\times n$ matrix $A$ has a completion $U V^t$ of rank $r$, then the symmetric $G$-partial matrix
\[
M = \begin{pmatrix}
\ast & A \\
A^t & \ast
\end{pmatrix}
\]
has the rank $r$ completion
\[
\begin{pmatrix}
U \\
V
\end{pmatrix}
\begin{pmatrix}
U^t & V^t
\end{pmatrix}.
\]
In particular, the generic fiber dimension of the morphism $\pi_G\vert_{V_{r+1}}\colon V_{r+1}\to \A^{\# E}$, where $V_{r+1}$ is the variety of symmetric matrices of rank at most $r+1$, is $(r+1)(m+n) - \binom{r+1}{2} - \# E$, because it is also dominant. Using the above inequality $r(m+n)-r^2 - \# E \geq 0$ together with $\binom{r+1}{2}\leq r^2$ for all $r\geq 1$, we get
\[
m+n + r(m+n) - \binom{r+1}{2} - \# E \geq m+n + (r(m+n) - r^2 - \# E) \geq m+n.
\]
Fixing the diagonal entries of a partial symmetric matrix in $\A^{\# E}$ imposes linear conditions on the fiber. Since a $G$-partial symmetric matrix consists of a partial matrix in $\A^{\# E}$ and $m+n$ diagonal entries, this dimension count suggests that we can complete a generic $G$-partial symmetric matrix to a symmetric matrix of rank $r+1$. Indeed, this is generically true, because the varieties in question are affine cones of quasi-projective varieties and so the claim follows by projective dimension theory.
\end{proof}

\end{document}